\documentclass{article}

\usepackage{amsmath,amssymb,amsfonts,amsthm}

\newtheorem{theorem}{Theorem}[section]

\newtheorem{proposition}[theorem]{Proposition}
\newtheorem{lemma}[theorem]{Lemma}
\numberwithin{equation}{section}

\theoremstyle{remark}
\newtheorem{remark}[theorem]{Remark}

\newcommand{\Ric}{\mathop{\mathrm{Ric}}\nolimits}
\newcommand{\Ad}{\mathop{\mathrm{Ad}}\nolimits}
\newcommand{\ad}{\mathop{\mathrm{ad}}\nolimits}
\newcommand{\pr}{\mathrm{pr}}
\newcommand{\tr}{\mathop{\mathrm{tr}}\nolimits}

\author{Artem Pulemotov\thanks{School of Mathematics and Physics, The
University of Queensland, St Lucia,~QLD 4072, Australia}\ \thanks{The author is the recipient of an Australian Research Council Discovery Early-Career Researcher Award~DE150101548.} \\
\small{\texttt{a.pulemotov@uq.edu.au}}}

\title{Metrics with prescribed Ricci curvature on homogeneous spaces}

\begin{document}

\maketitle

\begin{abstract}
Let $G$ be a compact connected Lie group and $H$ a closed subgroup of $G$. Suppose the homogeneous space $G/H$ is effective and has dimension~3 or higher. Consider a $G$-invariant, symmetric, positive-semidefinite, nonzero (0,2)-tensor field $T$ on $G/H$. Assume that $H$ is a maximal connected Lie subgroup of $G$. We prove the existence of a $G$-invariant Riemannian metric $g$ and a positive number $c$ such that the Ricci curvature of $g$ coincides with $cT$ on $G/H$. Afterwards, we examine what happens when the maximality hypothesis fails to hold.
\end{abstract}

\section{Introduction and statement of the main result}

The primary objective of the present paper is to produce a global existence theorem for Riemannian metrics with prescribed Ricci curvature on a broad class of compact homogeneous spaces. After stating and proving this theorem, we examine what happens when its key assumption is violated. Let us briefly discuss the history of the subject and describe our results in greater detail. 

Suppose $M$ is a smooth manifold. Finding a Riemannian metric $g$ on $M$ whose Ricci curvature $\Ric(g)$ coincides with a prescribed symmetric (0,2)-tensor field $T$ is a fundamental problem in geometric analysis. DeTurck proved the local existence of $g$ in the paper~\cite{DDT81} assuming $T$ was nondegenerate on $M$; see also~\cite[Chapter~5]{AB87} and~\cite[Section~6.5]{JK06}. Jointly with Goldschmidt, he obtained an analogous result for $T$ of constant rank in~\cite{DDTHG99}. That result required analyticity and several other conditions on~$T$. 

Many mathematicians have investigated the global existence of Riemannian metrics with prescribed Ricci curvature. The papers~\cite{ED02,RPKT09,RPLAMP15,AP13b} provide a snapshot of the recent progress on this topic. We refer to~\cite[Chapter~5]{AB87} and~\cite[Section~9.2]{TA98} for surveys of older work and to~\cite{GC07} for a sample of the research done in the Lorentzian setting. Most global existence results proven to date deal with open manifolds. However, there are a number of notable exceptions. For instance, Hamilton offered a series of theorems regarding metrics with prescribed Ricci curvature on spheres in~\cite{RH84}. DeTurck and Delano\"e obtained more general versions of one of those theorems in~\cite{DDT85} and~\cite{PhD03}.\footnote{It is unclear from the literature whether Hamilton and DeTurck were aware of each other's work when preparing the papers~\cite{RH84,DDT85}.} Note that many of the global existence results referenced above share a common feature. Namely, their proofs rely on various forms of the implicit and inverse function theorems.

As far as impact and applications are concerned,
DeTurck's work on the paper~\cite{DDT81} led him to the discovery of the DeTurck trick for the Ricci flow. Rubinstein showed in~\cite{YR08} that, under natural hypotheses, a sequence of K\"ahler metrics $(g_i)_{i=1}^\infty$ such that $\Ric(g_{i+1})$ equals $g_i$ for all $i\in\mathbb N$ must converge to a K\"ahler-Einstein metric. He also established a link between $(g_i)_{i=1}^\infty$ and discretisation of geometric flows. Subsequently, he conjectured (personal communication, 30 April 2013) that similar results held for Riemannian $(g_i)_{i=1}^\infty$, at least in some special situations.

Suppose the manifold $M$ is closed. In this case, instead of trying to prove the existence of a metric $g$ with $\Ric(g)$ equal to~$T$, one should search for a metric $g$ and a positive number~$c$ such that
\begin{align}\label{PRC}
\Ric(g)=cT.
\end{align}
This paradigm was originally proposed by Hamilton in~\cite{RH84} and DeTurck in~\cite{DDT85}. To explain it, we consider the problem of finding a metric on the 2-dimensional sphere $\mathbb S^2$ with prescribed positive-definite Ricci curvature~$T_{\mathbb S^2}$. According to the Gauss-Bonnet theorem and~\cite[Theorem~2.1]{RH84} (see also~\cite[Corollary~2.2]{DDT82}), such a metric exists if and only if the volume of $\mathbb S^2$ with respect to $T_{\mathbb S^2}$ equals~$4\pi$. Consequently, it is always possible to find $g$ and $c$ such that~\eqref{PRC} holds on $\mathbb S^2$ with $T$ replaced by $T_{\mathbb S^2}$. The value of $c$ is uniquely determined by $T_{\mathbb S^2}$. 
Hamilton suggests in~\cite[Section~1]{RH84} that the purpose of $c$ is to compensate for the invariance of the Ricci curvature under scaling of the metric.

Consider a compact connected Lie group $G$ and a closed connected subgroup $H<G$ such that the homogeneous space $G/H$ has dimension $n\ge3$. Let the manifold $M$ coincide with~$G/H$. For simplicity, we assume $G$ acts effectively on $M$. The following theorem is the main result of the present paper. We prove it in Section~\ref{sec_main_result}.

\begin{theorem}\label{thm_main}
Suppose $H$ is a maximal connected Lie subgroup of $G$. Let $T$ be a symmetric $G$-invariant (0,2)-tensor field on $M$. Assume $T$ is positive-semidefinite but not identically zero on $M$. There exist a $G$-invariant Riemannian metric $g$ and a number $c>0$ such that equation~\eqref{PRC} holds true.
\end{theorem}

Let us make a few remarks. For a comprehensive discussion of concrete examples of homogeneous spaces satisfying the hypotheses of Theorem~\ref{thm_main}, see~\cite[pages~185--186]{MWWZ86}. As far as conditions on $T$ are concerned, the majority of global existence results for metrics with prescribed Ricci curvature, including the results in~\cite{RH84,DDT85,PhD03}, require that the prescribed tensor field be positive- or negative-definite. Sometimes, this requirement is implicit. For instance, it is automatically satisfied if the prescribed tensor field is assumed to be ``close" to a non-Ricci-flat Einstein metric. Theorem~\ref{thm_main}, on the other hand, applies when $T$ is positive-semidefinite and nonzero. Our arguments break down when $T$ has mixed signature. Specifically, one cannot take advantage of Lemma~\ref{lem_est_scal_MT} in this case; see Remark~\ref{rem_T_pos} for details. Note that, according to Bochner's theorem, there are no $G$-invariant metrics on~$M$ with negative-semidefinite Ricci curvature.

Section~\ref{sec_H_not_max} explores the case where the isotropy representation of $M$ splits into two inequivalent irreducible summands and the maximality assumption on $H$ is violated. Our main result in this section, Proposition~\ref{prop_2sum}, provides a necessary and sufficient condition for the existence of a $G$-invariant metric $g$ and a positive number $c$ satisfying~\eqref{PRC}. When this condition holds, the pair $(g,c)$ is unique up to scaling of~$g$. Proposition~\ref{prop_2sum} implies that it is always possible to choose the tensor field~$T$ in such a way that $cT$ cannot be the Ricci curvature of any $G$-invariant metric for any number~$c$. Homogeneous spaces whose isotropy representations split into two irreducible summands were carefully studied by Dickinson and Kerr in the paper~\cite{WDMK08} and by He in the paper~\cite{CH12}; see also Buzano's work~\cite{MB14}.

\section{Proof of the main result}\label{sec_main_result}

The method we use to prove Theorem~\ref{thm_main} may be thought of as the method of Lagrange multipliers. Our reasoning stands on two pillars. The first one is the interpretation of metrics satisfying~\eqref{PRC} for some $c\in\mathbb R$ as critical points of a functional subject to a constraint. This interpretation, given by Lemma~\ref{lem_var} below, differs from the variational principle proposed by Hamilton in~\cite{RH84}. The second pillar is the technique invented by Wang and Ziller in~\cite{MWWZ86} to prove the existence of Einstein metrics on homogeneous spaces obeying the hypotheses of Theorem~\ref{thm_main}. This technique was further developed in~\cite{CB04,CBMWWZ04}. To make it work for our purposes will require a new estimate on the scalar curvature of a $G$-invariant metric on $M$. This estimate is given by Lemma~\ref{lem_est_scal_MT} below.

We begin with a few preparatory remarks. Let $\mathcal M$ be the space of $G$-invariant Riemannian metrics on~$M$. This space carries a natural smooth manifold structure; see, e.g.,~\cite[pages~6318--6319]{YNERVS07}. The scalar curvature $S(g)$ of a metric $g\in\mathcal M$ is constant on $M$. 
Therefore, we may interpret $S(g)$ as the result of applying a functional $S:\mathcal M\to\mathbb R$ to $g\in\mathcal M$. Note that $S$ is differentiable on $\mathcal M$; see, e.g.,~\cite[Corollary~7.39]{AB87}.

If the dimension of the space of $G$-invariant symmetric (0,2)-tensor fields on $M$ is equal to~1, then the assertion of Theorem~\ref{thm_main} is easy to prove. In this case, any metric $g\in\mathcal M$ satisfies formula~\eqref{PRC} for some $c\in\mathbb R$. Using Bochner's theorem (see~\cite[Theorem~1.84]{AB87}), one concludes that $c$ must be positive.
In the remainder of Section~\ref{sec_main_result}, we assume the dimension of the space of $G$-invariant symmetric (0,2)-tensor fields on $M$ is~2 or higher. Let $T$ be such a tensor field. Suppose $T$ is positive-semidefinite but not identically zero. Denote by $\tr_gT$ the trace of $T$ with respect to $g\in\mathcal M$. We write $\mathcal M_T$ for the space of all $g\in\mathcal M$ such that $\tr_gT=1$. The smooth manifold structure on $\mathcal M$ induces a smooth manifold structure on $\mathcal M_T$. In fact, $\mathcal M_T$ is a hypersurface in~$\mathcal M$. Denote by $S|_{\mathcal M_T}$ the restriction of the functional $S$ to~$\mathcal M_T$. As we will demonstrate below, $S|_{\mathcal M_T}$ attains its largest value when the conditions of Theorem~\ref{thm_main} are satisfied. Further variational properties of $S|_{\mathcal M_T}$ are discussed in Remark~\ref{rem_2sum_var} below.

\begin{lemma}\label{lem_var}
A Riemannian metric $g\in\mathcal M_T$ satisfies equation~\eqref{PRC} for some $c\in\mathbb R$ if and only if it is a critical point of $S|_{\mathcal M_T}$.
\end{lemma}

\begin{proof}
Fix a $G$-invariant symmetric $(0,2)$-tensor field $h$ on $M$. Let us compute the derivative $dS_g(h)$ for $g\in\mathcal M$. In order to do so, we consider the Einstein-Hilbert functional $E$ on the space $\mathcal M$. By definition,
\begin{align*}
E(g)=\int_MS(g)\,d\mu=S(g)\mu(M),\qquad g\in\mathcal M,
\end{align*}
where $\mu$ is the Riemannian volume measure corresponding to $g$. Consequently, the equality
\begin{align*}
dS_g(h)&=\frac1{\mu(M)}dE_g(h)-\frac{d(\mu(M))_g(h)}{\mu(M)^2}E(g),\qquad g\in\mathcal M,
\end{align*}
holds true. 
The well-known first variation formula for $E$ (see, e.g.,~\cite[Proposition~4.17]{AB87} or~\cite[Section~2.4]{BCPLLN06}) yields
\begin{align*}
dE_g(h)&=\int_M\left\langle\frac{S(g)}2g-\Ric(g),h\right\rangle\,d\mu=\left\langle\frac{S(g)}2g-\Ric(g),h\right\rangle\mu(M).
\end{align*}
The angular brackets here denote the scalar product in the tensor bundle over $M$ induced by $g$. It is easy to see that
\begin{align*}
d(\mu(M))_g(h)=\frac12\langle g,h\rangle\mu(M);
\end{align*}
cf.~\cite[Section~2.4]{BCPLLN06}. Thus, we have
\begin{align*}
dS_g(h)&=\left\langle\frac{S(g)}2g-\Ric(g),h\right\rangle-\frac{S(g)}2\langle g,h\rangle=-\langle\Ric(g),h\rangle,\qquad g\in\mathcal M.
\end{align*}
The space tangent to~$\mathcal M_T$ at any point consists of $G$-invariant symmetric (0,2)-tensor fields $h$ such that $\langle T,h\rangle=0$. Together with the above equality for $dS_g(h)$, this observation implies the assertion of the lemma.
\end{proof}

Our next objective is to state a formula for the scalar curvature of a $G$-invariant metric on $M$. We will use this formula to prove the existence of a critical point of $S|_{\mathcal M_T}$. First, we need to introduce more notation. Namely, let $\mathfrak g$ and $\mathfrak h$ be the Lie algebras of $G$ and $H$. Choose an
$\Ad(G)$-invariant scalar product $Q$ on $\mathfrak g$. Suppose $\mathfrak m$ is the $Q$-orthogonal complement of $\mathfrak h$ in~$\mathfrak g$. We standardly identify $\mathfrak m$ with the tangent space
of $M$ at~$H$. Consider a $Q$-orthogonal $\Ad(H)$-invariant decomposition
\begin{align}\label{m_decomp}
\mathfrak m=\mathfrak m_1\oplus\cdots\oplus\mathfrak m_s
\end{align}
such that $\Ad(H)|_{\mathfrak m_i}$ is irreducible for each $i=1,\ldots,s$. It is unique up to the order of summands if $\Ad(H)|_{\mathfrak m_i}$ is inequivalent to $\Ad(H)|_{\mathfrak m_k}$ whenever $i\ne k$. In the beginning of Section~\ref{sec_main_result}, we assumed the dimension of the space of $G$-invariant symmetric (0,2)-tensor fields on $M$ was~2 or higher. Therefore, $s$ must be greater than or equal to~2.

Our formula for the scalar curvature of a $G$-invariant metric will involve arrays of numbers, $(b_i)_{i=1}^s$ and $(\gamma_{ik}^l)_{i,k,l=1}^s$, associated with the scalar product $Q$ and the decomposition~\eqref{m_decomp}. To introduce the first one, suppose $B$ is the Killing form on the Lie algebra~$\mathfrak g$. For every $i=1,\ldots,s$, because $\Ad(H)|_{\mathfrak m_i}$ is irreducible, there exists a nonnegative $b_i$ such that 
\begin{align}\label{b_def}
B|_{\mathfrak m_i} = -b_iQ|_{\mathfrak m_i}.
\end{align}
To introduce the second array, fix a $Q$-orthonormal basis $(e_j)_{j=1}^n$ of $\mathfrak m$ adapted to the decomposition~\eqref{m_decomp}. Given $i,k,l=1,\dots,s$, define 
\begin{align}\label{gamma_def}
\gamma_{ik}^l=\sum Q([e_{\iota_i},e_{\iota_k}], e_{\iota_l})^2.
\end{align}
The sum is taken over all $\iota_i, \iota_k$ and $\iota_l$ such that $e_{\iota_i} \in \mathfrak m_i$, $e_{\iota_k} \in \mathfrak m_k$ and $e_{\iota_l} \in \mathfrak m_l$. Note that $\gamma_{ik}^l$ is independent of the choice of $(e_j)_{j=1}^n$ and symmetric in all three indices. Our further arguments will require the following property of the array $(\gamma_{ik}^l)_{i,k,l=1}^s$ due to Wang and Ziller (see the first paragraph in~\cite[Proof of Theorem~(2.2)]{MWWZ86}).

\begin{lemma}\label{lem_gamma}
Suppose $H$ is a maximal connected Lie subgroup of $G$. There exists a constant $a>0$ depending only on $G$, $H$ and $Q$ such that the following statement holds: for each non-empty proper subset $I\subset\{1,\ldots,s\}$, it is possible to find $i,k\in I$ and $l\notin I$ with $\gamma_{ik}^l\ge a$.
\end{lemma}

Recall that our objective is to state a formula for the scalar curvature of a $G$-invariant metric on $M$. Let $g$ lie in $\mathcal M$. Modifying the decomposition~\eqref{m_decomp} if necessary, we can write $g$ as the sum
\begin{align}\label{g_form_x}
g(X,Y)=\sum_{i=1}^sx_iQ(\pr_{\mathfrak m_i}X,\pr_{\mathfrak m_i}Y),\qquad X,Y\in\mathfrak m,
\end{align}
for some $x_i>0$; see~\cite[page~180]{MWWZ86}. The notation $\pr_{\mathfrak m_i}$ here means projection onto $\mathfrak m_i$. The scalar curvature of $g$ satisfies
\begin{align}\label{sc_curv_formula}
S(g)=\frac12\sum_{i=1}^s\frac{d_ib_i}{x_i}-\frac14\sum_{i,k,l=1}^s\gamma_{ik}^l\frac{x_l}{x_ix_k}
\end{align}
with $d_i$ the dimension of $\mathfrak m_i$. The reader will find the derivation of this formula in, e.g.,~\cite{MWWZ86} and~\cite[Chapter~7]{AB87}. Throughout Section~\ref{sec_main_result}, we assume $x_1\le\cdots\le x_s$ without loss of generality. The next two lemmas provide estimates on $S(g)$. The proof of the first one relies on~\eqref{sc_curv_formula}.

\begin{lemma}\label{lem_est_S}
Suppose $H$ is a maximal connected Lie subgroup of $G$. Then the formula
\begin{align*}
S(g)
\le\frac12\sum_{i=1}^s\frac{d_ib_i}{x_i}-\frac a{4(s-1)}\sum_{i=2}^s\frac{x_i}{x_{i-1}^2}
\end{align*}
holds true.
\end{lemma}

\begin{proof}
Let us choose $I=\{1\}$ in Lemma~\ref{lem_gamma}. We conclude that $\gamma_{i_1k_1}^{l_1}\ge a$ for $i_1,k_1$ equal to~1 and some $l_1$ between~2 and~$s$. Because $x_1\le\cdots\le x_s$, the estimate
\begin{align*}
S(g)&=\frac12\sum_{i=1}^s\frac{d_ib_i}{x_i}-\frac1{4(s-1)}\gamma_{i_1k_1}^{l_1}\frac{x_{l_1}}{x_{i_1}x_{k_1}}-\frac{s-2}{4(s-1)}\gamma_{i_1k_1}^{l_1}\frac{x_{l_1}}{x_{i_1}x_{k_1}}-\frac14\sum\gamma_{ik}^l\frac{x_l}{x_ix_k}
\\
&\le\frac12\sum_{i=1}^s\frac{d_ib_i}{x_i}-\frac a{4(s-1)}\frac{x_2}{x_1^2}-\frac{a(s-2)}{4(s-1)}\frac{x_{l_1}}{x_{i_1}x_{k_1}}-\frac14\sum\gamma_{ik}^l\frac{x_l}{x_ix_k}
\end{align*}
holds true. The sums without bounds here are taken over all the indices $i,k,l=1,\ldots,s$ with $(i,k,l)\ne(i_1,k_1,l_1)$. 

Choosing $I=\{1,2\}$ in Lemma~\ref{lem_gamma} yields $\gamma_{i_2k_2}^{l_2}\ge a$ for some $i_2,k_2$ equal to 1 or 2 and some $l_2$ between $3$ and $s$. If $(i_2,k_2,l_2)$ coincides with $(i_1,k_1,l_1)$, then
\begin{align*}
S(g)&
\le\frac12\sum_{i=1}^s\frac{d_ib_i}{x_i}-\frac a{4(s-1)}\frac{x_2}{x_1^2}
-\frac a{4(s-1)}\frac{x_{l_2}}{x_{i_2}x_{k_2}}-\frac{a(s-3)}{4(s-1)}\frac{x_{l_1}}{x_{i_1}x_{k_1}}-\frac14\sum\gamma_{ik}^l\frac{x_l}{x_ix_k}
\\ &\le
\frac12\sum_{i=1}^s\frac{d_ib_i}{x_i}-\frac a{4(s-1)}\left(\frac{x_2}{x_1^2}
+\frac{x_3}{x_2^2}\right)-\frac{a(s-3)}{4(s-1)}\frac{x_{l_1}}{x_{i_1}x_{k_1}}-\frac14\sum\gamma_{ik}^l\frac{x_l}{x_ix_k}.
\end{align*}
As before, the sums without bounds are taken over $i,k,l=1,\ldots,s$ with $(i,k,l)\ne(i_1,k_1,l_1)$. If $(i_2,k_2,l_2)$ differs from $(i_1,k_1,l_1)$, then
\begin{align*}
S(g)
&\le\frac12\sum_{i=1}^s\frac{d_ib_i}{x_i}-\frac a{4(s-1)}\frac{x_2}{x_1^2}
-\frac1{4(s-1)}\gamma_{i_2k_2}^{l_2}\frac{x_{l_2}}{x_{i_2}x_{k_2}}
\\ 
&\hphantom{=}~-\frac{s-2}{4(s-1)}\gamma_{i_2k_2}^{l_2}\frac{x_{l_2}}{x_{i_2}x_{k_2}} 
-\frac{a(s-2)}{4(s-1)}\frac{x_{l_1}}{x_{i_1}x_{k_1}}-\frac14\sum\gamma_{ik}^l\frac{x_l}{x_ix_k}
\\ &\le
\frac12\sum_{i=1}^s\frac{d_ib_i}{x_i}-\frac a{4(s-1)}\left(\frac{x_2}{x_1^2}
+\frac{x_3}{x_2^2}\right) \\
&\hphantom{=}~-\frac{a(s-2)}{4(s-1)}\left(\frac{x_{l_2}}{x_{i_2}x_{k_2}}+\frac{x_{l_1}}{x_{i_1}x_{k_1}}\right)-\frac14\sum\gamma_{ik}^l\frac{x_l}{x_ix_k}.
\end{align*}
Now the sums without bounds are over $i,k,l=1,\ldots,s$ with $(i,k,l)\ne(i_1,k_1,l_1)$ and $(i,k,l)\ne(i_2,k_2,l_2)$.

Consecutively choosing $I=\{1,2,\ldots,m\}$ for $m=1,2,\ldots,s-1$ in Lemma~\ref{lem_gamma} and arguing as above, we conclude that
\begin{align*}
S(g)\le
\frac12\sum_{i=1}^s\frac{d_ib_i}{x_i}-\frac a{4(s-1)}\sum_{i=2}^s\frac{x_i}{x_{i-1}^2}-\cdots,
\end{align*}
where the dots represent some nonnegative quantity. The required estimate on $S(g)$ follows immediately.
\end{proof}

Define the number $b>0$ by setting
\begin{align*}
b=-\inf B(X,X)+1,
\end{align*}
where the infimum is taken over the set of all $X\in\mathfrak m$ with $Q(X,X)=1$. It is clear that $b_i<b$ for all $i=1,\ldots,s$.

\begin{lemma}\label{lem_est_scal_MT}
Suppose $H$ is a maximal connected Lie subgroup of $G$. If $x_1\le\tau_1$ and $x_s\ge\tau_2$ for some positive numbers $\tau_1$ and $\tau_2$, then the estimate
\begin{align*}
S(g)\le\frac{bn}2x_1^{-1}-\alpha\bigg(x_1^{-\frac{2^{s-1}}{2^{s-1}-1}}+x_s^{\frac1{2^{s-1}-1}}\bigg)
\end{align*}
holds with the constant $\alpha>0$ depending only on $G$, $H$, $Q$, $\tau_1$ and $\tau_2$.
\end{lemma}

\begin{proof}
Lemma~\ref{lem_est_S} implies
\begin{align}\label{aux0}
S(g)\le\frac{bn}{2x_1}-\frac a{8(s-1)}S_1-\frac a{8(s-1)}S_2,
\end{align}
where
\begin{align*}
S_1=\frac{x_2}{\tau_1^2}+\sum_{i=3}^s\frac{x_i}{x_{i-1}^2},\qquad S_2=\sum_{i=2}^{s-1}\frac{x_i}{x_{i-1}^2}+\frac{\tau_2}{x_{s-1}^2}.
\end{align*}
We claim that
\begin{align}\label{aux1}
S_1\ge\alpha_1(s)x_s^{\frac1{2^{s-1}-1}}
\end{align}
for some $\alpha_1(s)>0$. The proof proceeds by induction in $s$. Indeed, it is obvious that estimate~\eqref{aux1} holds when $s=2$. Fix a natural number $m\ge2$ and assume this estimate holds for $s=m$. We will now prove it for $s=m+1$. By the inductive hypothesis,
\begin{align*}
S_1\ge\alpha_1(m)x_m^{\frac1{2^{m-1}-1}}+\frac{x_{m+1}}{x_m^2}
\end{align*}
with $\alpha_1(m)>0$. We treat the expression in the right-hand side as a function of $x_m$. This function attains its minimal value when 
\begin{align*}
\frac\partial{\partial x_m}\bigg(\alpha_1(m)x_m^{\frac1{2^{m-1}-1}}+\frac{x_{m+1}}{x_m^2}\bigg)=\frac{\alpha_1(m)}{2^{m-1}-1}x_m^{-\frac{2^{m-1}-2}{2^{m-1}-1}}-\frac{2x_{m+1}}{x_m^3}=0,
\end{align*}
that is,
\begin{align*}
x_m=\bigg(\frac{2^m-2}{\alpha_1(m)}\bigg)^{\frac{2^{m-1}-1}{2^m-1}}x_{m+1}^{\frac{2^{m-1}-1}{2^m-1}}.
\end{align*}
This minimal value is
\begin{align*}
\alpha_1(m)^{\frac{2^m-2}{2^m-1}}\Big((2^m-2)^{\frac1{2^m-1}}+(2^m-2)^{-\frac{2^m-2}{2^m-1}}\Big)x_{m+1}^{\frac1{2^m-1}}.
\end{align*}
Therefore, the estimate
\begin{align*}
S_1\ge\alpha_1(m+1)x_{m+1}^{\frac1{2^m-1}}
\end{align*}
must hold with 
\begin{align*}
\alpha_1(m+1)=\alpha_1(m)^{\frac{2^m-2}{2^m-1}}\Big((2^m-2)^{\frac1{2^m-1}}+(2^m-2)^{-\frac{2^m-2}{2^m-1}}\Big).\end{align*}
This concludes the proof of~\eqref{aux1}.

Let us demonstrate that
\begin{align}\label{aux2}
S_2\ge\alpha_2(s)x_1^{-\frac{2^{s-1}}{2^{s-1}-1}}
\end{align}
for some $\alpha_2(s)>0$. Again, we proceed by induction. The case $s=2$ is trivial. Given a natural $m\ge2$, assume estimate~\eqref{aux2} holds for $s=m$. We will prove this estimate for $s=m+1$. The inductive hypothesis implies
\begin{align*}
S_2\ge\frac{x_2}{x_1^2}+\alpha_2(m)x_2^{-\frac{2^{m-1}}{2^{m-1}-1}}
\end{align*}
with $\alpha_2(m)>0$. The derivative
\begin{align*}
\frac\partial{\partial x_2}\bigg(\frac{x_2}{x_1^2}+\alpha_2(m)x_2^{-\frac{2^{m-1}}{2^{m-1}-1}}\bigg)=\frac1{x_1^2}-\frac{\alpha_2(m)2^{m-1}}{2^{m-1}-1}x_2^{-\frac{2^m-1}{2^{m-1}-1}}
\end{align*}
is equal to~0 when
\begin{align*}
x_2=\bigg(\frac{2^{m-1}-1}{\alpha_2(m)2^{m-1}}\bigg)^{-\frac{2^{m-1}-1}{2^m-1}}x_1^{\frac{2^m-2}{2^m-1}}.
\end{align*}
This yields
\begin{align*}
S_2\ge\alpha_2(m+1)x_1^{-\frac{2^m}{2^m-1}},
\end{align*}
where 
\begin{align*}
\alpha_2(m+1)=\alpha_2(m)^{\frac{2^{m-1}-1}{2^m-1}}\Bigg(\bigg(\frac{2^{m-1}-1}{2^{m-1}}\bigg)^{-\frac{2^{m-1}-1}{2^m-1}}+\bigg(\frac{2^{m-1}-1}{2^{m-1}}\bigg)^{\frac{2^{m-1}}{2^m-1}}\Bigg).
\end{align*}
Thus, estimate~\eqref{aux2} is proven. The assertion of the lemma immediately follows from~\eqref{aux0}, \eqref{aux1} and~\eqref{aux2}.
\end{proof}

\begin{proof}[Proof of Theorem~\ref{thm_main}.]
We will show that the functional $S|_{\mathcal M_T}$ has a critical point. Lemma~\ref{lem_var} will then imply the existence of $g\in\mathcal M_T$ and $c\in\mathbb R$ satisfying equality~\eqref{PRC}. In the end, we will prove that $c$ must be positive.

Let us fix a $Q$-orthonormal basis in $\mathfrak m$. Given $u,v>0$, suppose $\mathcal M^{u,v}$ is the set of metrics $g\in\mathcal M$ such that
\begin{align*}
u\le g(X,X)\le v
\end{align*}
for all $X\in\mathfrak m$ with $Q(X,X)=1$. It is easy to see that $g\in\mathcal M$ lies in $\mathcal M^{u,v}$ if and only if the eigenvalues of the matrix of $g$ at $H$ in our fixed basis belong to the interval $[u,v]$. This observation implies that $\mathcal M^{u,v}$ is compact. The intersection $\mathcal M^{u,v}\cap\mathcal M_T$ is a closed subset of $\mathcal M^{u,v}$. Consequently, it must be compact as well. We will prove that, when $u$ and $v$ are chosen appropriately, the maximum of the functional $S|_{\mathcal M_T}$ over $\mathcal M^{u,v}\cap\mathcal M_T$ is also its global maximum. This will enable us to conclude $S|_{\mathcal M_T}$ has a critical point.

Suppose $g$ is a metric in $\mathcal M_T$. We write down formula~\eqref{g_form_x} and assume, without loss of generality, that $x_1\le\cdots\le x_s$. Lemma~\ref{lem_est_scal_MT} yields an estimate for $S(g)$. To produce this estimate, consider a $Q$-orthonormal basis $(e_j)_{j=1}^n$ of $\mathfrak m$ adapted to the decomposition~\eqref{m_decomp}. Define the numbers $\tau_1,\tau_2>0$ by setting
\begin{align}\label{postref1}
\tau_1=n\sup T(X,X),\qquad \tau_2=\frac1n\sum_{j=1}^nT(e_j,e_j),
\end{align}
where the supremum is taken over the set of all $X\in\mathfrak m$ with $Q(X,X)=1$. 
It is clear that $x_1$ cannot be greater than $\tau_1$. Indeed, the equality $\tr_gT=1$ implies
\begin{align*}
1=\sum_{j=1}^n\frac{T(e_j,e_j)}{g(e_j,e_j)}\le\frac{\tau_1}n\sum_{i=1}^s\frac{d_i}{x_i}\le\frac{\tau_1}{x_1}.
\end{align*}
Also, $x_s$ cannot be less than $\tau_2$. To see this, fix a natural number $p$ between~1 and $n$ such that 
\begin{align*}
T(e_p,e_p)=\max_{j=1,\ldots,n}T(e_j,e_j).
\end{align*}
It is clear that $T(e_p,e_p)\ge\tau_2$. 
Using the equality $\tr_gT=1$ one more time, we find
\begin{align*}
1=\sum_{j=1}^n\frac{T(e_j,e_j)}{g(e_j,e_j)}\ge\frac{T(e_p,e_p)}{g(e_p,e_p)}\ge\frac{\tau_2}{x_s}.
\end{align*}
Lemma~\ref{lem_est_scal_MT} implies the estimate
\begin{align}\label{postref2}
S(g)\le\frac{bn}2x_1^{-1}-\alpha x_1^{-\frac{2^{s-1}}{2^{s-1}-1}}
\end{align}
and the existence of $\tilde\alpha>0$ depending only on $G$, $H$, $Q$ and $T$ such that
\begin{align*}
S(g)\le\tilde\alpha-\alpha x_s^{\frac1{2^{s-1}-1}}.
\end{align*}
It is clear that $S(g)<0$ if
\begin{align*}
\inf g(X,X)=x_1<\bigg(\frac{2\alpha}{bn}\bigg)^{2^{s-1}-1}
\end{align*}
or
\begin{align*}
\sup g(X,X)=x_s>\bigg(\frac{\tilde\alpha}\alpha\bigg)^{2^{s-1}-1}.
\end{align*}
(The infimum and the supremum here are taken over all $X\in\mathfrak m$ with $Q(X,X)=1$.) Therefore, $S(g)<0$ if $g\in\mathcal M_T$ is outside the set $\mathcal M^{u,v}$ with
\begin{align*}
u=\bigg(\frac{2\alpha}{bn}\bigg)^{2^{s-1}-1},\qquad v=\bigg(\frac{\tilde\alpha}\alpha\bigg)^{2^{s-1}-1}.
\end{align*}
Obviously, the space $M$ carries a $G$-invariant 
metric with nonnegative scalar curvature. Multiplying this metric by a constant if necessary, we may assume that it lies in $\mathcal M_T$. Consequently, the compact set $\mathcal M^{u,v}\cap\mathcal M_T$ is nonempty, and the maximum of $S|_{\mathcal M_T}$ over $\mathcal M^{u,v}\cap\mathcal M_T$ is the maximum of $S|_{\mathcal M_T}$ over $\mathcal M_T$. It becomes clear that $S|_{\mathcal M_T}$ must have a critical point.

Lemma~\ref{lem_var} yields the existence of $g\in\mathcal M_T$ and $c\in\mathbb R$ satisfying~\eqref{PRC}. According to the Bochner theorem (see~\cite[Theorem~1.84]{AB87}), the group $G$ would have to be abelian if $c$ were nonpositive. However, since $s\ge2$, this would contradict Lemma~\ref{lem_gamma}. Thus, the formula $c>0$ must hold.
\end{proof}

\begin{remark}\label{rem_T_pos}
Theorem~\ref{thm_main} requires that the tensor field $T$ be positive-semidefinite. Let us explain the role this hypothesis plays in the proof. If $T$ had mixed signature, we would have been unable to establish the positivity of the number $\tau_2$ defined by~\eqref{postref1}. This would have prevented us from applying Lemma~\ref{lem_est_scal_MT} and obtaining estimate~\eqref{postref2}.
\end{remark}

\begin{remark}
Let $\mathcal M_1$ denote the set of all $g\in\mathcal M$ such that the volume of $M$ with respect to $g$ equals~1. According to~\cite[Theorem~(2.2)]{MWWZ86}, the functional $S|_{\mathcal M_1}$ is bounded from above and proper if and only if $H$ is a maximal connected Lie subgroup of $G$. It is natural to ask whether an analogous result holds for~$S|_{\mathcal M_T}$. The above proof of Theorem~\ref{thm_main} demonstrates that $S|_{\mathcal M_T}$ is bounded from above and proper if $H$ is a maximal connected Lie subgroup of $G$. We will not discuss the converse statement in the present paper.
\end{remark}

\begin{remark}
The conditions of Lemma~\ref{lem_est_scal_MT} are satisfied with $\tau_1=\tau_2=1$ for metrics in $\mathcal M_1$. By repeating the reasoning in Section~\ref{sec_main_result} with minor modifications, one concludes that, when $H$ is a maximal connected Lie subgroup of $G$, the map $S|_{\mathcal M_1}$ is bounded from above and proper. This yields a new proof of one of the implications in~\cite[Theorem~(2.2)]{MWWZ86}.
\end{remark}

\section{What if $H$ is not maximal?}\label{sec_H_not_max}

The purpose of this section is to explore the situation where the maximality hypothesis in Theorem~\ref{thm_main} is violated. We assume $\mathfrak m$ admits a $Q$-orthogonal $\Ad(H)$-invariant decomposition
\begin{align}\label{m2_decom}
\mathfrak m=\mathfrak m_1\oplus\mathfrak m_2
\end{align}
such that $\Ad(H)|_{\mathfrak m_1}$ and $\Ad(H)|_{\mathfrak m_2}$ are irreducible and inequivalent. In this case, the space of $G$-invariant symmetric (0,2)-tensor fields on $M$ is 2-dimensional. Let $T$ be such a tensor field. Assume $T$ is positive-semidefinite but not identically zero. The formula 
\begin{align*}
T(X,Y)=z_1Q(\pr_{\mathfrak m_1}X,\pr_{\mathfrak m_1}Y)+z_2Q(\pr_{\mathfrak m_2}X,\pr_{\mathfrak m_2}Y),\qquad X,Y\in\mathfrak m,
\end{align*}
holds for some $z_1,z_2\ge0$. The numbers $z_1$ and $z_2$ cannot equal~0 simultaneously.

Suppose the group $G$ has a connected proper Lie subgroup $K$ such that $H<K<G$ and $H\ne K$. Denote by $\mathfrak k$ the Lie algebra of $K$. It will be convenient for us to assume $\mathfrak k=\mathfrak h\oplus\mathfrak m_1$. This does not cause any loss of generality.

Formula~\eqref{gamma_def} defines an array of nonnegative constants, $\big(\gamma_{ik}^l\big)_{i,k,l=1}^2$, associated with the scalar product $Q$ and the decomposition~\eqref{m2_decom}. The equality $\gamma_{11}^2=0$ holds true. This equality follows from the inclusion $\mathfrak m_1\subset\mathfrak k$ and the fact that $\mathfrak k$ is orthogonal to $\mathfrak m_2$. We assume $\gamma_{22}^1\ne0$ when stating Proposition~\ref{prop_2sum} below. If $\gamma_{22}^1=0$, then all the metrics in $\mathcal M$ have the same Ricci curvature; see, e.g.,~\cite[Lemma~1.1]{JSPYS97}. Other consequences of this equality are discussed in~\cite[Proof of Theorem~(2.1)]{MWWZ86} and also~\cite[Remark~4.1.2]{MB14}.

Fix a $Q$-orthonormal basis $(w_j)_{j=1}^q$ of the Lie algebra $\mathfrak h$. Given $i=1,2$, the irreducibility of $\Ad(H)|_{\mathfrak m_i}$ implies the existence of a nonnegative constant $\zeta_i$ such that
\begin{align*}
-\Bigg(\sum_{j=1}^q\ad w_j\circ\ad w_j\Bigg)(X)=\zeta_iX
\end{align*}
for all $X\in\mathfrak m_i$. Note that $\zeta_i=0$ if and only if $\Ad(H)|_{\mathfrak m_i}$ is trivial. It is easy to verify that $\zeta_1$ and $\zeta_2$ cannot equal~0 simultaneously. According to~\cite[Lemma~(1.5)]{MWWZ86}, the formula
\begin{align}\label{Casimir}
d_ib_i=2d_i\zeta_i+\sum_{k,l=1}^2\gamma_{ik}^l,\qquad i=1,2,
\end{align}
holds with $b_i$ given by~\eqref{b_def} and $d_i$ the dimension of $\mathfrak m_i$.

\begin{proposition}\label{prop_2sum}
Assume $\gamma_{22}^1\ne0$. The following statements are equivalent:
\begin{enumerate}
\item
There exist a metric $g\in\mathcal M$ and a number $c>0$ such that the Ricci curvature of $g$ coincides with $cT$.
\item
The inequality
\begin{align}\label{2sum_cond}
\bigg(\zeta_2+\frac{\gamma_{22}^2}{4d_2}+\frac{\gamma_{22}^1}{d_2}\bigg)z_1>\bigg(\zeta_1+\frac{\gamma_{11}^1}{4d_1}\bigg)z_2
\end{align}
is satisfied.
\end{enumerate}
When these statements hold, the pair $(g,c)\in\mathcal M\times(0,\infty)$ is unique up to scaling of $g$.
\end{proposition}

\begin{proof}
Given a metric $g\in\mathcal M$, it is easy to see that
\begin{align*}
g(X,Y)=x_1Q(\pr_{\mathfrak m_1}X,\pr_{\mathfrak m_1}Y)+x_2Q(\pr_{\mathfrak m_2}X,\pr_{\mathfrak m_2}Y),\qquad X,Y\in\mathfrak m,
\end{align*}
for some $x_1,x_2>0$. The Ricci curvature of $g$ coincides with $cT$ if and only if
\begin{align*}
\frac{b_1}2-\frac{\gamma_{11}^1}{4d_1}-\frac{\gamma_{22}^1}{2d_1}+\frac{\gamma_{22}^1}{4d_1}\frac{x_1^2}{x_2^2}&=cz_1,
\\
\frac{b_2}2-\frac{\gamma_{22}^2}{4d_2}-\frac{\gamma_{22}^1}{2d_2}\frac{x_1}{x_2}&=cz_2;
\end{align*}
see, e.g.,~\cite[Lemma~1.1]{JSPYS97}. Using~\eqref{Casimir}, we rewrite these equalities as
\begin{align}\label{Pr_trans}
\zeta_1+\frac{\gamma_{11}^1}{4d_1}+\frac{d_2^2}{d_1\gamma_{22}^1}\left(\zeta_2+\frac{\gamma_{22}^2}{4d_2}+\frac{\gamma_{22}^1}{d_2}-cz_2\right)^2-cz_1&=0, \notag \\
\frac{2d_2}{\gamma_{22}^1}\left(\zeta_2+\frac{\gamma_{22}^2}{4d_2}+\frac{\gamma_{22}^1}{d_2}-cz_2\right)&=\frac{x_1}{x_2}.
\end{align}
Our objective is to show that~\eqref{2sum_cond} is a necessary and sufficient condition for the existence of $x_1,x_2,c>0$ satisfying~\eqref{Pr_trans}. This will prove the first assertion of the proposition. It will be clear from our arguments that the ratio $\frac{x_1}{x_2}$ and the number $c$ are uniquely determined by~\eqref{Pr_trans}. This fact implies the second assertion.

Suppose $z_2=0$. Then $x_1,x_2,c>0$ satisfying~\eqref{Pr_trans} obviously exist. Inequality~\eqref{2sum_cond} inevitably holds. The second line in~\eqref{Pr_trans} yields a unique value for $\frac{x_1}{x_2}$, and the first line determines~$c$. Thus, the proposition is proven. In what follows, assume $z_2\ne0$.

Let us show that the existence of $x_1,x_2,c>0$ satisfying~\eqref{Pr_trans} implies~\eqref{2sum_cond}. Transforming the first line in~\eqref{Pr_trans}, we find
\begin{align*}
\frac{d_2^2}{d_1\gamma_{22}^1}z_2^2c^2-\left(\eta_2z_2+z_1\right)c+\frac{d_1\gamma_{22}^1}{4d_2^2}(\eta_2^2+2\eta_1)=0,
\end{align*}
where
\begin{align*}
\eta_1=\frac{2d_2^2}{d_1\gamma_{22}^1}\left(\zeta_1+\frac{\gamma_{11}^1}{4d_1}\right),
\qquad
\eta_2=\frac{2d_2^2}{d_1\gamma_{22}^1}\left(\zeta_2+\frac{\gamma_{22}^2}{4d_2}+\frac{\gamma_{22}^1}{d_2}\right).
\end{align*}
This is a quadratic equation in $c$ with discriminant
\begin{align*}
D=(\eta_2z_2&+z_1)^2-(\eta_2^2+2\eta_1)z_2^2.
\end{align*}
Its solutions are given by the formula
\begin{align}\label{c_plus}
c=\frac{d_1\gamma_{22}^1\big(\eta_2z_2+z_1\pm\sqrt D\big)}{2d_2^2z_2^2}.
\end{align}
Substituting this into the second line in~\eqref{Pr_trans} yields
\begin{align}\label{ratio_def}
\frac{x_1}{x_2}&=-\frac{d_1}{d_2z_2}\big(z_1\pm\sqrt D\big).
\end{align}
Because the ratio $\frac{x_1}{x_2}$ is positive, the expression in the right-hand side must be positive. As a consequence, we obtain
\begin{align}\label{cond_aux1}
\eta_1z_2<\eta_2z_1,
\end{align}
which is equivalent to~\eqref{2sum_cond}. 

The above arguments demonstrate that, when $x_1,x_2,c>0$ satisfying~\eqref{Pr_trans} exist, the ratio $\frac{x_1}{x_2}$ and the number $c$ are given by~\eqref{ratio_def} and~\eqref{c_plus}. In both formulas, the sign before the square root must be a minus. Thus, $\frac{x_1}{x_2}$ and $c$ are determined uniquely.

Let us now assume that~\eqref{2sum_cond} holds. We will produce $x_1,x_2,c>0$ satisfying~\eqref{Pr_trans}. This will complete the proof of the proposition. Observe that the discriminant $D$ is inevitably positive. Indeed, formula~\eqref{cond_aux1}, which is equivalent to~\eqref{2sum_cond}, yields
\begin{align}\label{est_D_last}
D=z_1^2+2\eta_2z_1z_2-2\eta_1z_2^2>z_1^2\ge0.
\end{align}
We define $c$ by~\eqref{c_plus} with a minus in front of the square root. It is easy to check that $c$ is positive. The first equality in~\eqref{Pr_trans} holds true. Next, we define
\begin{align*}
x_1=-\frac{d_1}{d_2z_2}\big(z_1-\sqrt D\big),\qquad x_2=1.
\end{align*}
The positivity of $x_1$ follows from~\eqref{est_D_last}. An elementary computation shows that the second equality in~\eqref{Pr_trans} holds true.
\end{proof}

\begin{remark}\label{rem_2sum_var}
Rather than arguing as above, one may prove Proposition~\ref{prop_2sum} by exploiting Lemma~\ref{lem_var} and examining the functional $S|_{\mathcal M_T}$. Note that this functional is bounded from above unless $z_1=0$ and $\zeta_1+\gamma_{11}^1>0$.
\end{remark}

\section*{Acknowledgements}

I express my gratitude to Yanir Rubinstein and Godfrey Smith for the stimulating discussions on metrics with prescribed Ricci curvature. I am also thankful to Wolfgang Ziller for suggesting several useful references to me.

\end{document}